\documentclass[a4paper]{amsart}
\usepackage{amssymb, enumerate}
\usepackage[all]{xy}
\usepackage{hyperref, aliascnt}

\usepackage[capitalise, nameinlink, noabbrev, nosort]{cleveref}

\theoremstyle{plain}
\newtheorem{lma}{Lemma}[section]
\crefname{lma}{Lemma}{Lemmata}
\newtheorem{thm}[lma]{Theorem}
\crefname{thm}{Theorem}{Theorems}
\newtheorem{cor}[lma]{Corollary}
\crefname{cor}{Corollary}{Corollaries}
\newtheorem{prp}[lma]{Proposition}
\crefname{prp}{Proposition}{Propositions}

\theoremstyle{definition}

\crefname{pgr}{Paragraph}{Paragraphs}
\newtheorem{dfn}[lma]{Definition}
\crefname{dfn}{Definition}{Definitions}

\theoremstyle{remark}
\newtheorem{rmk}[lma]{Remark}
\crefname{rmk}{Remark}{Remarks}

\newtheorem{qst}[lma]{Question}
\crefname{qst}{Question}{Questions}

\theoremstyle{plain}

\newcounter{IntroCount}

\newtheorem{thmIntro}[IntroCount]{Theorem}
\crefname{thmIntro}{Theorem}{Theorems}

\def\today{\number\day\space\ifcase\month\or   January\or February\or
   March\or April\or May\or June\or   July\or August\or September\or
   October\or November\or December\fi\   \number\year}

\newcommand{\andSep}{\,\,\,\text{ and }\,\,\,}
\newcommand{\CatCu}{\ensuremath{\mathrm{Cu}}}
\newcommand{\CuSgp}{$\CatCu$-sem\-i\-group}

\newcommand{\filter}{{\mathcal{U}}}
\newcommand{\calZ}{\mathcal{Z}}
\newcommand{\calD}{\mathcal{D}}
\newcommand{\calU}{\mathcal{U}}
\newcommand{\calO}{\mathcal{O}}
\newcommand{\NN}{{\mathbb{N}}}
\newcommand{\CC}{{\mathbb{C}}}
\newcommand{\RR}{{\mathbb{R}}}
\newcommand{\Cpct}{{\mathcal{K}}}
\newcommand{\id}{{\mathrm{id}}}
\newcommand{\dist}{{\mathrm{dist}}}
\newcommand{\ca}{$C^*$-algebra}
\newcommand{\stHom}{$\ast$-ho\-mo\-mor\-phism}
\newcommand{\stIso}{$\ast$-isomorphism}

\DeclareMathOperator{\Sep}{Sep}
\DeclareMathOperator{\Cu}{Cu}

\title{Centrally pure C*-algebras}
\date{\today}

\author{Francesc Perera, Hannes Thiel, Eduard Vilalta}

\address{
F.~Perera, 
Departament de Matem\`{a}tiques,
Universitat Aut\`{o}noma de Barcelona,
\linebreak 08193 Bellaterra, Barcelona, Spain, and
Centre de Recerca Matem\`atica, Edifici Cc, Campus de Bellaterra,  08193 Cerdanyola del Vall\`es, Barcelona, Spain}
\email[]{francesc.perera@uab.cat}
\urladdr{https://mat.uab.cat/web/perera}

\address{Hannes~Thiel, 
Department of Mathematical Sciences, Chalmers University of Technology and the University of Gothenburg, SE-412 96 Gothenburg, Sweden}
\email{hannes.thiel@chalmers.se}
\urladdr{www.hannesthiel.org}

\address{Eduard~Vilalta, 
Department de Matem\`{a}tiques, Universitat Polit\`{e}cnica de Catalunya, Diagonal 647, Barcelona.}
\email{eduard.vilalta@upc.edu}
\urladdr{www.eduardvilalta.com}

\thanks{
FP and EV were partially supported by the Spanish State Research Agency throught the grant  PID2023-147110NB-I00 and by the Comissionat per Universitats i Recerca de la Ge\-ne\-ralitat de Ca\-ta\-lu\-nya (grant No.\ 2021-SGR-01015). FP was also supported by the Spanish State Research Agency through the Severo Ochoa and María de Maeztu Program for Centers and Units of Excellence in R\&D (CEX2020-001084-M).
HT and EV were partially supported by the Knut and Alice Wallenberg Foundation (KAW 2021.0140).
HT was partially supported by the Swedish Research Council project grant 2024-04200.
}

\subjclass[2020]%
{Primary
46L05; 
Secondary
19K14, 
46L80, 
46L85. 
}
\date{\today}

\begin{document}

\begin{abstract}
We show that a separable 
\ca{} $A$ is $\calZ$-stable 
if and only if its uncorrected central sequence algebra $A' \cap A_\filter$ is pure, if and only if Kirchberg's central sequence algebra $F(A)$ is pure.

More generally, we show that a \ca{} $A$ is separably $\calZ$-stable if and only if the relative central sequence algebra $B' \cap A_\filter$ is pure for every separable subalgebra $B \subseteq A_\filter$.
\end{abstract}

\maketitle

\section{Introduction}

Pureness and $\calZ$-stability have emerged as decisive regularity properties in the structure and classification theory of \ca{s}.~Their central roles became apparent in the development of the Elliott program, which seeks to classify simple, nuclear \ca{s} with an invariant encoding $K$-theoretic and tracial data.
Toms \cite{Tom08ClassificationNuclear} constructed examples of such algebras with the same invariant that are nevertheless nonisomorphic, since one is $\calZ$-stable (i.e., it absorbs tensorially the Jiang-Su algebra $\calZ$), while the other is not.~These groundbreaking examples showed that additional regularity assumptions are indispensable in the classification program.

Subsequent work has established that $\calZ$-stability is exactly the additional regularity assumption required.
Indeed, building on decades of research, a series of landmark results \cite{TikWhiWin17QDNuclear, EllGonLinNiu25ClassFinDecompRankII, GonLinNiu20ClassifZstable2} (see also the recent approach \cite{CarGabSchTikWhi23arX:ClassifHom1} and the survey \cite{Whi23AbstractClassif}) established that simple, separable, nuclear, $\calZ$-stable \ca{s} satisfying the Universal Coefficient Theorem (UCT) are classified up to isomorphism by their Elliott invariants.
Moreover, it is conceivable that the UCT condition is superfluous, or that it can be replaced by a classification framework based on $\mathrm{KK}$-theory \cite{Sch24arX:KKRigSimpleNucCAlgs}.

Pureness, introduced by Winter in his foundational work \cite{Win12NuclDimZstable} on the structure of simple, nuclear \ca{s}, is defined as the combination of strict comparison and almost divisibility in the Cuntz semigroup (see~\cref{sec:purecentral}, and also \cite{AntPerThiVil24arX:PureCAlgs, PerThiVil25arX:ExtPureCAlgs} for a more systematic studty of pure \ca{s}).~It is closely related to $\calZ$-stability: R{\o}rdam \cite{Ror04StableRealRankZ} showed that $\calZ$-stability implies pureness in general, while the famous Toms–Winter conjecture \cite{Win18ICM} predicts that for simple, separable, nonelementary, nuclear \ca{s} $\calZ$-stability and strict comparison are equivalent, and this has been verified in many important cases \cite{Win12NuclDimZstable, Sat12arx:TraceSpace, KirRor14CentralSeq, TomWhiWin15ZStableFdBauer, Thi20RksOps}.
Since pureness lies between strict comparison and $\calZ$-stability, the conjecture in particular predicts that pureness and $\calZ$-stability coincide in this setting.
More broadly, the nonsimple Toms–Winter conjecture \cite{AntPerThiVil24arX:PureCAlgs} proposes that a separable, nowhere scattered, nuclear \ca{} is $\calZ$-stable if and only if it is pure.~This has been verified in some cases, for example in \cite[Theorems 7.10, 7.15]{RobTik17NucDimNonSimple}.

Beyond the nuclear setting, pureness continues to play an important role.~For instance, reduced free group \ca{s} are pure \cite{AmrGaoElaPat24arX:StrCompRedGpCAlgs} although not $\calZ$-stable, showing that pureness can provide regularity where $\calZ$-stability does not apply.
Moreover, the solution to the rank problem in the stable rank one setting \cite{Thi20RksOps, AntPerRobThi22CuntzSR1} shows that for simple, nonelementary \ca{s} of stable rank one —and more generally for nowhere scattered \ca{s} of stable rank one— pureness is equivalent to strict comparison.


In this paper we establish another close connection between $\calZ$-stability and pureness.~For a free ultrafilter $\mathcal{U}$ on $\NN$, denote by $A_\mathcal{U}$ the ultrapower of $A$.~Our main result is:

\begin{thmIntro}[see \ref{prp:CharZStable}]
Let $A$ be a separable \ca{}. 
Then the following are equivalent:
\begin{enumerate}[{\rm (1)}]
    \item $A$ is $\calZ$-stable.
    \item The central sequence algebra $A' \cap A_\mathcal{U}$ is pure.
    \item Kirchberg’s central sequence algebra $F(A):=(A' \cap A_\mathcal{U})/(A^\perp \cap A_\mathcal{U})$ is pure.
\end{enumerate}
\end{thmIntro}

This shows that $\calZ$-stability can be characterized as \emph{central pureness}.~It also suggests that pureness may be interpreted as a noncentral analog of $\calZ$-stability.  

We also obtain a nonseparable generalization. Recall that a \ca{} $A$ is \emph{separably $\calZ$-stable} if any separable sub-\ca{} $B_0$ is contained in a separable, $\calZ$-stable sub-\ca{} $B$; see  \cref{dfn:SepDStable} and the comments preceding it.

\begin{thmIntro}[see \ref{prp:CharSepZStable}]
Let $A$ be a \ca{}. 
Then the following are equivalent:
\begin{enumerate}[{\rm (1)}]
    \item $A$ is separably $\calZ$-stable.
    \item The relative central sequence algebra $B' \cap A_\mathcal{U}$ is pure for every separable subalgebra $B \subseteq A_\mathcal{U}$.
    \item The quotient $(B' \cap A_\mathcal{U})/(B^\perp \cap A_\mathcal{U})$ is pure for every separable subalgebra $B \subseteq A_\mathcal{U}$.
\end{enumerate}
\end{thmIntro}

Along the way (\cref{sec:DAbs}), we collect and streamline a number of well-known results on absorption of strongly self-absorbing \ca{s}, some of which had not previously appeared in the literature in full generality.~In the final section (\cref{sec:CentralDiv}), we investigate divisibility in central sequence algebras.~Our results show that good central comparison follows from good central divisibility, so that $\calZ$-stability is already detected by central divisibility properties.~In fact, for a \ca{} to be $\calZ$-stable it suffices for its central sequence algebra to satisfy $n$-almost divisibility for some fixed $n$ (\cref{prp:CentralDiv}). It remains open whether even weaker divisibility properties, such as functional divisibility, are enough (\cref{qst:CentralFuncDiv}).

\section{Absorption of a strongly self-absorbing \texorpdfstring{$C^*$}{C*}-algebra}
\label{sec:DAbs}

In this section, we present basic results that characterize when a separable \ca{} tensorially absorbs a strongly self-absorbing \ca{} $\calD$; 
see \cref{prp:CharDStable}.
These results are well-known, but certain aspects seem to not have appeared in the literature so far, or only under additional assumptions like unitality. We also give characterizations of separable $\calD$-stability in the not necessarily separable case; see \cref{prp:CharSepDStable}.


Throughout the paper, $\filter$ denotes a free ultrafilter on $\NN$.
Given a \ca{}~$A$, we use $\prod_\NN A$ to denote the \ca{} of bounded sequences in $A$, and the \emph{ultrapower} of $A$, denoted by $A_\filter$, is defined as the quotient $\prod_\NN A / c_\filter(A)$, where $c_\filter(A)$ is the closed, two-sided ideal given by
\[
c_\filter(A) := \left\{ (a_n)_n \in \prod_\NN A : \lim_{n\to\filter}\|a_n\|=0 \right\}.
\]
We view $A$ as a subalgebra of $A_\filter$ via the natural embedding $A \to A_\filter$ given by mapping an element $a \in A$ to the image of the constant sequence $(a,a,\ldots)\in\prod_\NN A$ in~$A_\filter$.
The \emph{central sequence algebra} of $A$ is defined as the commutant of $A$ inside~$A_\filter$, denoted by $A'\cap A_\filter$.
Note that the image of a sequence $(a_n)_n \in \prod_\NN A$ in $A_\filter$ belongs to $A'\cap A_\filter$ if and only if 
\[
\lim_{n\to\filter} \| a_nb - ba_n \| = 0
\]
for every $b \in A$.

The \emph{annihilator} of $A$ inside $A_\filter$ (denoted by $A^\perp \cap A_\filter$) is the set of elements $b\in A_\filter$ satisfying $bA=\{0\}=Ab$.
This is a closed, two-sided ideal in $A' \cap A_\filter$, and \emph{Kirchberg's central sequence algebra} is $F(A) := (A' \cap A_\filter) / (A^\perp \cap A_\filter)$; 
see \cite[Definition~1.1]{Kir06CentralSeqPI}.
To stress the difference between $F(A)$ and $A' \cap A_\filter$, we sometimes refer to the latter as the \emph{uncorrected central sequence algebra}.

If $A$ is unital, then $A^\perp \cap A_\filter=\{0\}$ and we get $F(A) = A' \cap A_\filter$, but in general~$F(A)$ is a proper quotient of $A' \cap A_\filter$.
If $A$ is $\sigma$-unital, then $F(A)$ is unital.

A unital, separable \ca{} $\mathcal{D}$ is \emph{strongly self-absorbing} if $\calD \neq \CC$ and there is a $\ast$-isomorphism $\calD \to \calD \otimes \calD$ which is approximately unitarily equivalent to $\id_\calD \otimes 1_\calD$;
see \cite[Definition~1.3]{TomWin07ssa}.
It is known that strongly self-absorbing \ca{s} are automatically simple and nuclear, and the only known examples are the Jiang-Su algebra $\calZ$, the Cuntz algebras~$\calO_2$ and $\calO_\infty$, UHF-algebras of infinite type, and the tensor product of $\calO_\infty$ with an UHF-algebra of infinite type. For $\calD$ strongly self-absorbing, a \ca{} $A$ is said to be \emph{$\calD$-stable} provided $A \cong A \otimes \calD$.

The next result is folklore. 
Its converse holds whenever $A$ is separable;
see \cref{prp:CharDStable}.

\begin{lma}
\label{prp:DStable-Nonsep}
Let $\calD$ be a strongly self-absorbing \ca, and let $A$ be a (not necessarily separable) unital \ca{} that is $\calD$-stable.
Then there exists a unital embedding $\calD \to A' \cap A_\filter$.
\end{lma}
\begin{proof}
The sequence of unital embeddings
\[
\varphi_n \colon \calD \to A \otimes \calD \otimes \calD \otimes \cdots, \quad
d \mapsto 1_A \otimes 1_\calD \otimes \cdots_n \otimes 1_\calD \otimes d \otimes 1_\calD \otimes \cdots
\]
is approximately central in the sense that 
\[
\lim_{n\to\infty} \| \varphi_n(d)a - a\varphi_n(d) \| = 0.
\]
for every $d \in \calD$ and $a \in A \otimes \calD \otimes \calD \otimes \cdots$.
Using that
\[
A 
\cong A \otimes \calD
\cong A \otimes \calD \otimes \calD \otimes \cdots,
\]
the sequence $(\varphi_n)_n$ induces a unital embedding $\calD \to A' \cap A_\filter$.
\end{proof}

The next result is well-known to experts.~The equivalence of~(1) and~(3) was shown in \cite[Theorem~2.2]{TomWin07ssa} and is based on R{\o}rdam's approximate intertwining theorem \cite[Theorem~7.2.2]{Ror02Classification}; 
see also \cite[Theorem~2.5]{KirRor15CentralSeqCharacters}.
To the best of our knowledge, the equivalence of~(1) and~(4) for non-unital \ca{s} has only appeared in the literature for tensorial absorption of the Jiang-Su algebra in \cite[Proposition~5.1]{Naw12PicardGpStablyProjlessCAlg}, and we adapt the argument of Nawata to obtain the general result.
A crucial ingredient is Winter's result that strongly self-absorbing \ca{s} are $K_1$-injective \cite{Win11ssaZstable};
see also Kirchberg's result on tensorial absorption \cite[Proposition~4.11]{Kir06CentralSeqPI}. For a \ca{} $A$, we denote by $M(A)$ its multiplier algebra.

\begin{thm}
\label{prp:CharDStable}
Let $A$ be a separable \ca, and let $\calD$ be a strongly self-absorbing \ca.
Then the following are equivalent:
\begin{enumerate}[{\rm (1)}]
\item
The \ca{} $A$ is $\calD$-stable.
\item
There is a $\calD$-stable \ca{} $B$ with $A \subseteq B \subseteq M(A)$ and $1_{M(A)} \in B$.
\item
There is a unital \stHom{} $\calD \to A' \cap M(A)_\filter$.
\item
There is a unital \stHom{} $\calD \to F(A)$.
\end{enumerate}
\end{thm}
\begin{proof}
We show the implications `(1)$\Rightarrow$(2)$\Rightarrow$(3)$\Rightarrow$(4)$\Rightarrow$(1)'.

(1)$\Rightarrow$(2):
Assume that $A$ is $\calD$-stable, and fix a \stIso{} $\varphi \colon A \to A \otimes \calD$.
This induces a \stIso{} $M(\varphi) \colon M(A) \to M(A \otimes \calD)$.
Let $\iota \colon A \to M(A)$ and $\iota' \colon A \otimes \calD \to M(A \otimes \calD)$ denote the natural inclusions.
Then $\iota\otimes\id_\calD$ identifies $A \otimes \calD$ with an essential, closed ideal in $M(A) \otimes \calD$.
By the universal property of multiplier algebras, we obtain a (unique) unital, injective \stHom{} $\psi \colon M(A) \otimes \calD \to M(A \otimes \calD)$ such that $\iota' = \psi\circ(\iota\otimes\id_\calD)$.
The situation is shown in the following commutative diagram:

\[
\xymatrix{
A \ar[rr]^{\iota} \ar[d]^{\varphi}_{\cong}
& & M(A) \ar[d]^{M(\varphi)}_{\cong} \\
A \otimes \calD \ar[r]^-{\iota\otimes\id_\calD}
& M(A) \otimes \calD \ar[r]^-{\psi}
& M(A \otimes \calD).
}
\]
The desired algebra is $B:= M(\varphi)^{-1}(\psi(M(A)\otimes\calD))$.


(2)$\Rightarrow$(3):
Let $B$ be a $\calD$-stable \ca{} with $A \subseteq B \subseteq M(A)$ and $1_{M(A)} \in B$.
Applying \cref{prp:DStable-Nonsep}, we get a unital \stHom{} $\calD \to B' \cap B_\filter$.
Composing with the natural inclusion $A' \cap B_\filter \subseteq A' \cap M(A)_\filter$, we obtain the desired map
\[
\calD \to A' \cap B_\filter
\subseteq A' \cap M(A)_\filter.
\]


(3)$\Rightarrow$(4):
Let $(e_n)_n$ be a contractive, approximate identity in $A$.
This induces a (well-defined) contractive, linear map $M(A)_\filter \to A_\filter$ by mapping $[(x_n)_n] \in M(A)_\filter$ to $[(e_nx_n)_n] \in A_\filter$, which restricts to a map $\pi_0 \colon A' \cap M(A)_\filter \to A' \cap A_\filter$.

While $\pi_0$ is not necessarily multiplicative, we have $\pi_0(xy)-\pi_0(x)\pi_0(y) \in A^\perp$ for all $x,y \in A' \cap M(A)_\filter$.
It follows that the composition of $\pi_0$ with the quotient map $A' \cap A_\filter \to F(A)$ 
is a \stHom{} $\pi\colon A' \cap M(A)_\filter \to F(A)$.
One can show that $\pi$ is surjective, and moreover independent of the choice of a contractive, approximate identity in~$A$.

Now, if there is a unital \stHom{} $\calD \to A' \cap M(A)_\filter$, then by composing with $\pi$ we obtain a unital \stHom{} $\calD \to F(A)$.


(4)$\Rightarrow$(1):
By assumption, there exists a unital \stHom{} $\alpha\colon \calD \to F(A)$.
To verify $A \cong A\otimes\calD$, we consider the subalgebra $A \otimes 1$ of $M(A\otimes\calD)$, which induces embeddings $A\otimes 1 \subseteq (A \otimes 1)_\filter \subseteq M(A\otimes\calD)_\filter$.
We will show that there exists a sequence of unitaries $(u_n)_n$ in $(A\otimes 1)' \cap M(A\otimes\calD)_\filter$ such that 
\[
\lim_{n\to\infty} \dist\big( u_nbu_n^*, (A \otimes 1)_\filter \big) = 0
\]
for all $b \in A\otimes\calD \subseteq M(A\otimes\calD) \subseteq M(A\otimes\calD)_\filter$.
It then follows from R{\o}rdam's approximate intertwining theorem \cite[Proposition~7.2.1]{Ror02Classification} that $A \cong A\otimes\calD$.

To find the unitaries in $(A\otimes 1)' \cap M(A\otimes\calD)_\filter$, we consider the quotient
\[
C := \left( \big( A\otimes 1 \big)' \cap M(A\otimes\calD)_\filter \right) / \left( \big( A\otimes 1 \big)^\perp \cap M(A\otimes\calD)_\filter \right),
\]
and we let $\pi\colon ( A\otimes 1 )' \cap M(A\otimes\calD)_\filter \to C$ denote the quotient map.
The subalgebra $1 \otimes \calD \subseteq M(A\otimes\calD) \subseteq M(A\otimes\calD)_\filter$ lies in the commutant of $A\otimes 1$, which defines a unital \stHom{}
\[
\beta_0 \colon \calD \xrightarrow{\cong} 1 \otimes \calD \to (A \otimes 1)' \cap M(A\otimes\calD)_\filter.
\]
Composed with $\pi$, we obtain a unital \stHom{} $\beta := \pi\circ\beta_0 \colon \calD \to C$.

We can also interpret $\alpha$ as a map with codomain $C$.~Specifically, denote by $\tilde{\pi}\colon A' \cap M(A)_\filter \to F(A)$ the surjective \stHom{} constructed in the proof of the implication `(3)$\Rightarrow$(4)'.
One can show that the kernel of $\tilde{\pi}$ is $A^\perp \cap M(A)_\filter$, which establishes a natural \stIso{}
\[ 
F(A) \cong (A' \cap M(A)_\filter) / (A^\perp \cap M(A)_\filter).
\]
This allows us to view $F(A)$ as a natural subalgebra of $C$ and, via this identification, we have $\alpha\colon D\to C$.
The situation is shown in the following diagram:
\[
\small
\xymatrix@R15pt@C-15pt{ 
& & \calD \ar[d]^{\alpha} \\
(A \otimes 1)' \cap (A \otimes 1)_\filter \ar[r] \ar@{^{(}->}[d]
& \big( (A\otimes 1)' \cap (A \otimes 1)_\filter \big) / \big( (A\otimes 1)^\perp \cap (A \otimes 1)_\filter \big) \ar@{^{(}->}[d] \ar@{}[r]|<<<{\cong}
& F(A) \\
(A \otimes 1)' \cap M(A\otimes\calD)_\filter \ar@{->>}[r]^<<<{\pi}
& \big( (A\otimes 1)' \cap M(A\otimes\calD)_\filter \big) / \big( (A\otimes 1)^\perp \cap M(A\otimes\calD)_\filter \big) \ar@{}[r]|<<<{=:}
& C \\
\calD \cong 1 \otimes \calD \ar[u]^{\beta_0} \ar[ur]_{\beta}
}
\]

One verifies that the maps $\alpha,\beta\colon \calD \to C$ have commuting images, and since $\calD$ is nuclear we obtain a unital \stHom{} $\alpha\otimes\beta\colon \calD\otimes\calD \to C$.
It follows from the definition of strong self-absorption that there is a sequence $(w_n)_n$ of unitaries in $\calD\otimes\calD$ such that $\lim_{n\to\infty} w_n(x \otimes y)w_n^* = y \otimes x$ for all $x,y \in \calD$ (one says that $\calD$ has approximately inner flip).
By \cite[Proposition~1.13]{TomWin07ssa}, the unitaries $w_n$ may be chosen to be trivial in $K_1(\calD\otimes\calD)$.
By \cite[Theorem~3.1]{Win11ssaZstable}, every strongly self-absorbing \ca{} is itself $\calZ$-stable, and therefore it is $K_1$-injective.~This implies that the unitaries $w_n$ lie in $\calU_0(\calD\otimes\calD)$, the connected component of the unit in the unitary group.

Setting $v_n := (\alpha\otimes\beta)(w_n)$, we thus have a sequence $(v_n)_n$ of unitaries in $\calU_0(C)$ such that
\[
\lim_{n\to\infty} v_n \alpha(x)\beta(y) v_n^* = \beta(y) \alpha(x)
\]
for all $x,y \in \calD$.
Since unitaries in the connected component of the identity can be lifted along quotient maps, we can find unitaries $u_n$ in $(A\otimes 1)' \cap M(A\otimes\calD)_\filter$ with $\pi(u_n)=v_n$.

To show that $(u_n)_n$ have the desired properties, let $a \in A$ and $d \in \calD$.
Since $\alpha(\calD) \subseteq \pi( (A \otimes 1)' \cap (A \otimes 1)_\filter )$, we can pick $x \in (A \otimes 1)' \cap (A \otimes 1)_\filter$ such that $\pi(x)=\alpha(d)$.
We then have
\[
\pi\big( u_n\beta_0(d)u_n^* \big)
= v_n\beta(d)v_n^*
\to \alpha(d)
= \pi(x).
\]
Since the kernel of $\pi$ is the annihilator of $A\otimes 1$, we deduce that
\[
u_n(a \otimes d)u_n^*
= u_n(a \otimes 1)\beta_0(d)u_n^*
= (a \otimes 1)u_n\beta_0(d)u_n^* 
\to (a \otimes 1)x
\in (A \otimes 1)_\filter .
\]
This shows that $\lim_{n\to\infty} \dist\big( u_nbu_n^*, (A \otimes 1)_\filter \big) = 0$ for every elementary tensor $b=a\otimes d$ in $A\otimes\calD$, and consequently for every $b \in A\otimes\calD$.
\end{proof}

While the definition of $\calD$-stability makes sense for arbitrary \ca{s}, it is less useful for nonseparable \ca{s} since even natural candidates such as the ultrapower of $\calD$ fail to be $\calD$-stable.
Instead, the model-theoretically natural concept is to require that `sufficiently many' separable sub-\ca{s} are $\calD$-stable.
Variations of this concept have appeared under different names in the literature. 
We follow the terminology of Schafhauser \cite[Definition~1.4]{Sch20SubSimpAF};
see also \cite[Proposition~1.6]{FarSza24arX:CoronaSSA}.

\begin{dfn}
\label{dfn:SepDStable}
Let $\calD$ be a strongly self-absorbing \ca.
A \ca{} $A$ is said to be \emph{separably $\calD$-stable} if for every separable sub-\ca{} $B_0 \subseteq A$ there exists a separable sub-\ca{} $B \subseteq A$ that is $\calD$-stable and such that $B_0 \subseteq B$.
\end{dfn}

We use $\Sep(A)$ to denote the collection of separable sub-\ca{s} of a given \ca{} $A$. 
Restricting the model-theoretic definition to our case of interest, a family $F \subseteq \Sep(A)$ is a \emph{club} if it is \emph{cofinal} (that is, for every $B_0 \in \Sep(A)$ there exists $B\in F$ such that $B_0 \subseteq B$) and \emph{$\sigma$-complete} (that is, for every countable directed subset $F'\subseteq F$ we have that $\overline{\bigcup_{B \in F'}B} \in F$).

Let $\calD$ be a strongly self-absorbing \ca{}.
Using that $\calD$-stability passes to sequential inductive limits of separable \ca{s} \cite[Corollary~3.4]{TomWin07ssa}, we see that a \ca{} $A$ is separably $\calD$-stable if and only if the separable, $\calD$-stable sub-\ca{s} of $A$ form a club in $\Sep(A)$.

Toms and Winter \cite[Section~3]{TomWin07ssa} proved permanence properties of $\calD$-stability among separable \ca{s}, in particular passage to quotients, to hereditary sub-\ca{s}, and to extensions.
Farah and Szabó \cite{FarSza24arX:CoronaSSA} showed that analogous permanence properties hold for separable $\calD$-stability, which we record for later use.

\begin{lma}
\label{prp:PermanenceSepDStable}
Let $\calD$ be a strongly self-absorbing \ca.~The following statements hold:
\begin{enumerate}[{\rm (1)}]
\item
If $0 \to I \to E \to B \to 0$ is an extension of \ca{s}, then $E$ is separably $\calD$-stable if and only if $I$ and $B$ are separably $\calD$-stable.
\item
If $A$ is separably $\calD$-stable, and $B \subseteq A$ is a hereditary sub-\ca{}, then~$B$ is separably $\calD$-stable.
\end{enumerate}
\end{lma}
\begin{proof}
The first statement was proved in \cite[Lemma~1.14]{FarSza24arX:CoronaSSA}, and the second statement was shown in \cite[Proposition~1.12]{FarSza24arX:CoronaSSA}.
\end{proof}

The following concept was introduced in \cite[Definition~1.11]{FarSza24arX:CoronaSSA}.

\begin{dfn}
\label{dfn:Dsat}
Let $\calD$ be a unital \ca.
We say that a unital \ca{}~$A$ is \emph{$\calD$-saturated} if for every separable sub-\ca{} $B \subseteq A$ there is a unital $\ast$-embedding $\calD\to B'\cap A$.
\end{dfn}

\begin{thm}
\label{prp:CharSepDStable}
Let $\calD$ be a strongly self-absorbing \ca, and let $A$ be a \ca.
Then the following are equivalent:
\begin{enumerate}[{\rm (1)}]
\item
The \ca{} $A$ is separably $\calD$-stable.
\item
The ultrapower $A_\filter$ is separably $\calD$-stable.
\item
The relative commutant $B' \cap A_\filter$ is separably $\calD$-stable for every separable sub-\ca{} $B \subseteq A_\filter$.
\item
The quotient $(B' \cap A_\filter)/(B^\perp \cap A_\filter)$ is $\calD$-saturated for every separable sub-\ca{} $B \subseteq A_\filter$.
\item
There exists a unital $\ast$-homomorphism $\calD \to (B' \cap A_\filter)/(B^\perp \cap A_\filter)$ for every separable sub-\ca{} $B \subseteq A_\filter$.
\end{enumerate}
If $A$ is $\sigma$-unital, then the above statements are also equivalent to
\begin{enumerate}[{\rm (1)}]
\setcounter{enumi}{5}
\item
The multiplier algebra $M(A)$ is separably $\calD$-stable.
\item
The \ca{} $M(A)_\filter$ is $\calD$-saturated.
\end{enumerate}
\end{thm}
\begin{proof}
We show the implications `(1)$\Leftrightarrow$(2)' and `(1)$\Leftrightarrow$(5)' and `(3)$\Rightarrow$(4)$\Rightarrow$(5)$\Rightarrow$(3)' and `(1)$\Leftrightarrow$(6)' and `(6)$\Rightarrow$(7)$\Rightarrow$(6)'.

(1)$\Leftrightarrow$(2):
It follows from \cite[Corollary~1.10]{FarSza24arX:CoronaSSA} that separable $\calD$-stability is axiomatizable;
see also \cite[Theorem~2.5.2(21)]{FarHarLupRobTikVigWin21ModelThy}.
Since axiomatizable properties pass to ultrapowers and ultraroots (\cite[Theorem 2.4.1]{FarHarLupRobTikVigWin21ModelThy}), we see that~(1) and~(2) are equivalent.

\smallskip

(1)$\Leftrightarrow$(5):
We use that by \cite[Corollary~1.10(2) and Definition 1.8]{FarSza24arX:CoronaSSA}, the \ca{} $A$ is separably $\calD$-stable if and only if $A$ satisfies condition (1) in \cite[Proposition 1.7]{FarSza24arX:CoronaSSA}, namely, for any sub-\ca{} $B_0$ of $A$, there is a unital $\ast$-homomorphism $\calD\to (B_0'\cap A_\infty)/(B_0^\perp\cap A_\infty)$. 

For the forward implication, assuming that $A$ is separably $\calD$-stable, let $B \subseteq A_\filter$ be a separable sub-\ca.
We consider the natural quotient map $\pi\colon A_\infty \to A_\filter$ from the sequence algebra $A_\infty := \prod_\NN A / c_0(A)$ where $c_0(A)$ is the closed ideal of sequences $(a_n)_n$ in $A$ with $\lim_{n\to\infty}\|a_n\|=0$.
Choose a separable sub-\ca{} $B_0 \subseteq A_\infty$ with $\pi(B_0)=B$.
By \cite[Proposition~1.7 (1)]{FarSza24arX:CoronaSSA}, there exists a unital \stHom{} $\calD \to (B_0' \cap A_\infty)/(B_0^\perp \cap A_\infty)$.
The map $\pi$ restricts to a map $B_0' \cap A_\infty \to B' \cap A_\filter$, which induces a unital \stHom{} from $(B_0' \cap A_\infty)/(B_0^\perp \cap A_\infty)$ to $(B' \cap A_\filter)/(B^\perp \cap A_\filter)$.
We obtain the desired map as the composition
\[
\calD \to (B_0' \cap A_\infty)/(B_0^\perp \cap A_\infty)
\to (B' \cap A_\filter)/(B^\perp \cap A_\filter).
\]

For the backward implication, we notice that statement~(5) immediately implies condition~(3) in \cite[Proposition~1.7]{FarSza24arX:CoronaSSA}, which then implies that $A$ is separably $\calD$-stable.

\smallskip

(3)$\Rightarrow$(4):
To verify~(4), let $B \subseteq A_\filter$ be a separable sub-\ca. 
Set $Q:= (B' \cap A_\filter)/(B^\perp \cap A_\filter)$ and let $\pi\colon B' \cap A_\filter \to Q$ denote the quotient map.
To show that~$Q$ is $\calD$-saturated, let $C \subseteq Q$ be separable.
We need find a unital \stHom{} $\calD \to C' \cap Q.$

Choose a separable sub-\ca{} $C_0 \subseteq B' \cap A_\filter$ with $\pi(C_0)=C$.
Since the \ca{} $E$ generated by $C_0$ and $B$ is separable, it follows from the assumption that $E' \cap A_\filter$ is separably $\calD$-stable, and hence so is the quotient $\pi(E' \cap A_\filter)$ by \cref{prp:PermanenceSepDStable}.
Since $\pi(E' \cap A_\filter)$ is unital and separably $\calD$-stable, it admits a unital \stHom{} from $\calD$, and we obtain the desired map as the composition
\[
\calD 
\to \pi(E' \cap A_\filter)
\subseteq \pi(C_0)' \cap \pi(B ' \cap A_\filter)
= C' \cap Q.
\]

(4)$\Rightarrow$(5):
This follows using that every unital, $\calD$-saturated \ca{} admits a unital \stHom{} from $\calD$.

\smallskip

(5)$\Rightarrow$(3):
Assume that~(5) holds.
With a similar argument as in the implication `(3)$\Rightarrow$(4)', we see that $(B' \cap A_\filter)/(B^\perp \cap A_\filter)$ is $\calD$-saturated (and consequently separably $\calD$-stable) for every separable sub-\ca{} $B \subseteq A_\filter$.

To verify~(3), let $B \subseteq A_\filter$ be a separable sub-\ca, and consider the extension
\[
0 \to B^\perp \cap A_\filter \to B' \cap A_\filter \to (B' \cap A_\filter)/(B^\perp \cap A_\filter) \to 0.
\]
We have already shown that (5) implies~(2), which shows that $A_\filter$ is separably $\calD$-stable.
Since $B^\perp \cap A_\filter$ is a hereditary sub-\ca{} of $A_\filter$, it follows from \cref{prp:PermanenceSepDStable}(2) that $B^\perp \cap A_\filter$ is separably $\calD$-stable.
Further, as noted above, the quotient $(B' \cap A_\filter)/(B^\perp \cap A_\filter)$ is also separably $\calD$-stable.
Now it follows from \cref{prp:PermanenceSepDStable}(1) that $B' \cap A_\filter$ is separably $\calD$-stable.

\smallskip

(1)$\Leftrightarrow$(6):
Assuming that $A$ is $\sigma$-unital, the equivalence of~(1) and~(6) is shown in \cite[Theorem~B]{FarSza24arX:CoronaSSA}.

\smallskip

(6)$\Rightarrow$(7):
Assume that $M(A)$ is separably $\calD$-stable.
To show that $M(A)_\filter$ is $\calD$-saturated, let $B \subseteq M(A)_\filter$ be separable.
Applying the equivalence of~(1) and~(3) for $M(A)$, it follows that $B' \cap M(A)_\filter$ is separably $\calD$-stable, and in particular admits a unital \stHom{} $\calD \to B' \cap M(A)_\filter$. Namely, if $B_0$ is a unital separable sub-\ca{} of $B'\cap M(A)_\calU$, then there is a separable $\calD$-stable sub-\ca{} $\tilde{B}_0$ with $B_0\subseteq \tilde{B}_0$.

\smallskip

(7)$\Rightarrow$(6):
Assume that $M(A)_\filter$ is $\calD$-saturated. Note that every (unital) $\calD$-saturated \ca{} $B$ is separably $\calD$-stable. To check this, let $C_0$ be a separable sub-\ca{} of $B$. Since $B$ is $\calD$-saturated, there is a unital $\ast$-embedding $\varphi\colon\calD\to C_0'\cap B$. Now $\calD_0:=\varphi(\calD)$ is separable and commutes with $C_0$, so if we set $C_1=C^*(C_0\cup\calD_0)\subseteq B$ we get a separable sub-\ca{} that contains $C_0$ and $\calD\otimes C_0\to C_1$, given by $d\otimes c\mapsto \varphi(d)c$ is an isomorphism.

Therefore, it follows that $M(A)_\filter$ is separably $\calD$-stable.
Applying the equivalence of~(1) and~(2) for $M(A)$, it follows that $M(A)$ is separably $\calD$-stable.
\end{proof}

Applying \cref{prp:CharSepDStable} for a separable \ca{}, we obtain:

\begin{cor}
\label{prp:CharDStable2}
Let $A$ be a separable \ca, and let $\calD$ be a strongly self-absorbing \ca.
Then the following are equivalent:
\begin{enumerate}[{\rm (1)}]
\item
The \ca{} $A$ is $\calD$-stable.
\item
The algebra $M(A)_\filter$ is $\calD$-saturated.
\item
The algebra $A' \cap M(A)_\filter$ is $\calD$-saturated.
\item
The (uncorrected) central sequence algebra $A' \cap A_\filter$ is separably $\calD$-stable.
\item
Kirchberg's central sequence algebra $F(A)$ is $\calD$-saturated.
\end{enumerate}
\end{cor}
\begin{proof}
We show the implications `(1)$\Leftrightarrow$(2)$\Rightarrow$(3)$\Rightarrow$(1)' and `(1)$\Leftrightarrow$(4)' and `(1)$\Leftrightarrow$(5)'.

(1)$\Leftrightarrow$(2):
This follows from the equivalence of~(1) and~(7) in \cref{prp:CharSepDStable}.

\smallskip

(2)$\Rightarrow$(3):
In general, if $B$ is a $\calD$-saturated \ca{} and $C \subseteq B$ is a separable sub-\ca{}, then $C' \cap B$ is $\calD$-saturated.~For, if $E\subseteq C'\cap B$ is a separable sub-\ca{}, then $C^*(E\cup C)$ is also separable and by assumption there is a unital $\ast$-embedding $\calD\to C^*(E\cup C)'\cap B\subseteq E'\cap C'\cap B$. Now, since $A$ is separable, we see that~(2) implies~(3).

\smallskip

(3)$\Rightarrow$(1):
Assuming that $A' \cap M(A)_\filter$ is $\calD$-saturated, it follows that there is a unital \stHom{} $\calD \to A' \cap M(A)_\filter$.
Applying the implication `(3)$\Rightarrow$(1)' in \cref{prp:CharDStable}, we see that $A$ is $\calD$-stable.

\smallskip

(1)$\Leftrightarrow$(4):
Since $A$ is separable, the forward implication follows from the implication `(1)$\Rightarrow$(3)' in \cref{prp:CharSepDStable}.
For the backward implication, assume that $A' \cap A_\filter$ is separably $\calD$-stable.
Since $F(A)$ is a quotient of $A' \cap A_\filter$, and using \cref{prp:PermanenceSepDStable}(1), we deduce that $F(A)$ is separably $\calD$-stable and therefore admits a unital \stHom{} $\calD \to F(A)$.
Applying the implication `(4)$\Rightarrow$(1)' in \cref{prp:CharDStable}, we see that $A$ is $\calD$-stable.

\smallskip

(1)$\Leftrightarrow$(5):
Since $A$ is separable, the forward implication follows from the implication `(1)$\Rightarrow$(4)' in \cref{prp:CharSepDStable}.
Conversely, if $F(A)$ is $\calD$-saturated, then it admits a unital \stHom{} $\calD \to F(A)$, which by the implication `(4)$\Rightarrow$(1)' in \cref{prp:CharDStable} shows that $A$ is $\calD$-stable.
\end{proof}

\section{Pureness of central sequence algebras}
\label{sec:purecentral}
In this section, we prove the main result of this paper (\cref{prp:CharZStable}):
a separable \ca{}~$A$ is $\calZ$-stable if and only if its central sequence algebra $A' \cap A_\filter$ is pure.
In \cref{prp:CharSepZStable}, we obtain a similar characterization of separable $\calZ$-stability (for not necessarily separable \ca{s}) in terms of pureness of the relative central sequence algebras $B' \cap A_\filter$ for separable sub-\ca{s} $B \subseteq A_\filter$. 

Pureness is defined as the combination of strict comparison and almost divisibility, which can be thought of as good comparison and divisibility properties in the Cuntz semigroup; see below for the precise definitions. From this point of view, \cref{prp:CharZStable} shows that a separable \ca{} is $\calZ$-stable if and only if its central sequence algebra enjoys good comparison and good divisibility properties.

The relation $<_s$ on an ordered monoid is defined by setting $x<_sy$ if there exists $n \in \NN$ such that $(n+1)x \leq ny$. Then, by definition, $A$ has \emph{strict comparison} if $x<_s y$ in $\Cu(A)$ implies $x\leq y$. This is also referred to by saying that $\Cu(A)$ is \emph{almost unperforated}. One says that $A$ is \emph{$n$-almost divisible} if, for any pair $x'x\in \Cu(A)$ with $x'\ll x$, $k\in\NN$, there is $y\in\Cu(A)$ with $ky\leq x$ and $x'\leq (n+1)(k+1)y$. The \ca{} $A$ is \emph{almost divisible} if it is $0$-almost divisible.

\smallskip

We now record the following observation.

\begin{prp}
\label{prp:SepZStablePure}
Every separably $\calZ$-stable \ca{} is pure.
\end{prp}
\begin{proof}
R{\o}rdam showed in \cite{Ror04StableRealRankZ} that every $\calZ$-stable \ca{} has almost unperforated Cuntz semigroup. 
It follows from \cite[Theorem~5.35]{AraPerTom11Cu} that it is also almost divisible, and hence pure; 
see \cite[Proposition~5.2]{AntPerThiVil24arX:PureCAlgs}.
Now let $A$ be a separably $\calZ$-stable \ca{}.
Then the family of separable, $\calZ$-stable sub-\ca{s} of $A$ forms an inductive system (indexed over itself) whose inductive limit is (isomorphic to) $A$.
By \cite[Theorem~3.8]{PerThiVil25arX:ExtPureCAlgs}, pureness of \ca{s} is separably determined (a term also introduced in \cite[Definition~2.2]{PerThiVil25arX:ExtPureCAlgs}), so in particular it passes to inductive limits, which shows that $A$ is pure.
\end{proof}

For $n \geq 1$, the \emph{dimension-drop algebra} $Z_{n,n+1}$ is defined as the sub-\ca{} of $C([0,1],M_n\otimes M_{n+1})$ consisting of the continuous functions $f \colon [0,1] \to M_n\otimes M_{n+1}$ that satisfy
\[
f(0)\in M_n\otimes 1_{n+1}, \andSep 
f(1)\in M_{n+1}\otimes 1_n.
\]

In \cite[Proposition~5.1]{RorWin10ZRevisited}, R{\o}rdam and Winter showed that a unital \ca{}~$A$ of stable rank one admits a unital \stHom{} $Z_{n,n+1} \to A$ if and only if the unit can be `almost' divided by $n$ in the Cuntz semigroup $\Cu(A)$ in the sense that there exists a Cuntz class $x \in \Cu(A)$ such that $nx \leq [1] \leq (n+1)x$.

In \cite[Theorem~3.6]{DadTom10Ranks}, Dadarlat and Toms proved the same result for separable, unital \ca{s} with nonempty tracial state space and with strict comparison of positive elements by tracial states.
The next result generalizes the result of Dadarlat and Toms by removing the assumptions of separability and of nonempty tracial state space, and by weakening strict comparison of positive elements by tracial states to strict comparison of positive elements by quasitraces.
A proof can essentially be found in \cite[Proposition~4.9]{NgRob16CommutatorsPureCa}, noting that the argument can be adapted to only require almost divisibility of the unit. 
We reproduce the (adapted) argument here for the convenience of the reader.

\begin{prp}
\label{prp:Z23-to-Pure}
Let $A$ be a unital \ca{} with strict comparison (of positive elements by quasitraces). Then, the following are equivalent:
\begin{enumerate}[{\rm(1)}]
\item 
The element $[1]$ is almost divisible in $\Cu(A)$. 
\item 
For each $n \geq 1$, there exists a unital \stHom{} $Z_{n,n+1} \to A$.
\end{enumerate}
\end{prp}
\begin{proof}
Assuming~(2), let us verify~(1).~To show that $[1]$ is almost divisible, let $n \geq 1$.
We need to find $y \in \Cu(A)$ such that $ny \leq [1] \leq (n+1)y$.
By assumption, there is a unital \stHom{} $\varphi \colon Z_{n,n+1} \to A$.
By \cite[Lemma~4.2]{Ror04StableRealRankZ}, there exists a positive element $e \in Z_{n,n+1}$ such that $n[e] \leq [1] \leq (n+1)[e]$ in $\Cu(Z_{n,n+1})$.
Now $y:=[\varphi(e)]$ has the desired properties.

\smallskip

Assuming~(1), let us prove~(2).~For $a,b\in A_+$, write $a\simeq b$ whenever $a=xx^*$ and $b=x^*x$ for some $x\in A$. We will apply \cite[Proposition~5.1]{RorWin10ZRevisited}, which states that there exists a unital \stHom{} $Z_{n,n+1} \to A$ whenever one can find pairwise orthogonal elements $a_1,\ldots ,a_n \in A_+$ such that $a_i\simeq a_j$ for each pair $i,j$ and $1-\sum_i a_i\precsim (a_1-\varepsilon )_+$ for some $\varepsilon >0$ (where $\precsim$ denotes Cuntz subequivalence).  

To verify~(2), let $n \geq 1$ and set $m=2n$. 
By (1), there exists $c \in (A\otimes\Cpct)_+$ such that
$m[c] \leq [1] \leq (m+1)[c]$.
Since $[1] \ll [1]$, we can pick $\varepsilon>0$ such that $[1] \leq (m+1)[(c-\varepsilon )_+]$.
We thus have
\[
m[c] \leq [1] \leq (m+1)[(c-\varepsilon )_+].
\]

Applying \cite[Lemma~2.4]{RobRor13Divisibility} (or, rather, its proof) to the first inequality, one obtains pairwise orthogonal elements $c_1,\ldots ,c_m\in A_+$ such that
\[
(c-\tfrac{\varepsilon}{2})_+\simeq c_1\simeq \ldots \simeq c_m.
\]
Let $f \colon \RR \to [0,1]$ be a continuous function that takes the value $0$ on $(-\infty,0]$ and that takes the value $1$ on $[\tfrac{\varepsilon}{2},\infty)$.
Using continuous functional calculus, set $d_j' := (c_j-\tfrac{\varepsilon}{2})_+ $ and $d_j := f(c_j)$, and note that then $d_j'=d_j'd_j$ for $j=1,\ldots,m$.
Since $d_j$ belongs to $\overline{c_jAc_j}$, and since $c_1,\ldots,c_m$ are pairwise orthogonal, it follows that $d_1,\ldots,d_m$ are pairwise orthogonal as well.

For each $j=1,\ldots,m$, we claim that $d_j \simeq f( (c-\tfrac{\varepsilon}{2})_+ )$ and $d_j' \simeq (c-\varepsilon)_+$. 
Since $(c-\tfrac{\varepsilon}{2})_+ \simeq c_j$, we can choose $x_j \in A$ such that $(c-\tfrac{\varepsilon}{2})_+ = x_jx_j^*$ and $c_j = x_j^*x_j$.
Let $x_j = v_j|x_j|$ be the polar decomposition of $x_j$ in $A^{**}$, and set
\[
y_j := f\big( (c-\tfrac{\varepsilon}{2})_+ \big)^{\frac{1}{2}} v_j, \andSep
z_j := ( c-\varepsilon )_+^{\frac{1}{2}} v_j.
\]
Since $f\big( (c-\tfrac{\varepsilon}{2})_+ \big)^{\frac{1}{2}}$ and $( c-\varepsilon )_+^{\frac{1}{2}}$ belong to $\overline{A|x_j^*|}$ and $|x_j^*|v_j=x_j \in A$, we deduce that $y_j$ and $z_j$ belong to $A$.
We have
\[
y_jy_j^*
= f\big( (c-\tfrac{\varepsilon}{2})_+ \big), \andSep
z_jz_j^*
= ( c-\varepsilon )_+.
\]

Note that $v_j^*(c-\tfrac{\varepsilon}{2})_+v_j =v_j^*x_jx_j^*v_j=|x_j||x_j|= c_j$.
Since the map $b \mapsto v_j^*bv_j$ defines a $\ast$-isomorphism from $\overline{(c-\tfrac{\varepsilon}{2})_+A(c-\tfrac{\varepsilon}{2})_+}$ onto $\overline{c_jAc_j}$, and since applying functional calculus is compatible with \stHom{s}, we deduce that
\[
y_j^*y_j
= v_j^* f\big( (c-\tfrac{\varepsilon}{2})_+ \big) v_j
= f\big( v_j^*(c-\tfrac{\varepsilon}{2})_+v_j \big) 
= f(c_j)
= d_j
\]
and
\begin{align*}
z_j^*z_j
&= v_j^* ( c-\varepsilon )_+ v_j \\
&= v_j^* \big( (c-\tfrac{\varepsilon}{2})_+ - \tfrac{\varepsilon}{2} \big)_+  v_j \\
&= \big( v_j^*(c-\tfrac{\varepsilon}{2})_+v_j - \tfrac{\varepsilon}{2} \big)_+ \\
&= (c_j-\tfrac{\varepsilon}{2})_+
= d_j'.
\end{align*}
It follows that
\[
d_j
= y_j^*y_j
\simeq y_jy_j^*
= f\big( (c-\tfrac{\varepsilon}{2})_+ \big), \andSep
d_j'
= z_j^*z_j
\simeq z_jz_j^*
= ( c-\varepsilon )_+,
\]
as claimed.

\smallskip

Given a quasitrace $\tau$ on $A$, we define $d_\tau \colon (A\otimes\Cpct)_+\to[0,\infty]$ by $d_\tau (a):=\lim_n \tau (a^{1/n})$.
Using at the second step that $1-\sum_{j=1}^m d_j$ is orthogonal to each $d_1',\ldots,d_m'$, and using at the last step that $(c-\varepsilon)_+$ is in particular Cuntz equivalent to $d_1'$, one gets
\begin{align*}
d_\tau \left( 1-\sum_{j=1}^m d_j \right) + m d_\tau (d_1') 
&= d_\tau \left( 1-\sum_{j=1}^m d_j \right) + \sum_{j=1}^m d_\tau (d_j')
\leq d_\tau (1) \\
&\leq (m+1) d_\tau((c-\varepsilon)_+)
= (m+1) d_\tau(d_1').
\end{align*}
Since $d_\tau(d_1')<\infty$, we can cancel $m d_\tau (d_1')$ and obtain
\[
d_\tau \left( 1-\sum_{j=1}^m d_j \right) 
\leq  d_\tau(d_1')
= \frac{1}{2}d_\tau(d_1') + \frac{1}{2}d_\tau(d_2')
= \frac{1}{2}d_\tau(d_1'+d_2').
\] 
Pair the elements $d_j$ by setting $a_1=d_1+d_2$, \ldots, $a_n=d_{2n-1}+d_{2n}$.
We have 
\[
d_1' 
= (c_1-\tfrac{\varepsilon}{2})_+
\precsim (f(c_1)-\tfrac{\varepsilon}{2})_+ 
= (d_1-\tfrac{\varepsilon}{2})_+
\]
and similarly $d_2' \precsim (d_2-\tfrac{\varepsilon}{2})_+$.
Using that $d_1$ and $d_2$ are orthogonal, we deduce that
\[
d_1' + d_2'
\precsim (d_1-\tfrac{\varepsilon}{2})_+ + (d_2-\tfrac{\varepsilon}{2})_+
= (a_1-\tfrac{\varepsilon}{2})_+.
\]
We thus have 
\[
d_\tau\left( 1-\sum_{j=1}^n a_j \right)
= d_\tau \left( 1-\sum_{j=1}^m d_j \right) 
\leq \frac{1}{2}d_\tau(d_1'+d_2')
\leq \frac{1}{2}d_\tau\big( (a_1-\tfrac{\varepsilon}{2})_+ \big).
\]

Using strict comparison, one deduces from the previous functional inequality that $1-\sum_{j=1}^n a_j \precsim (a_1-\tfrac{\varepsilon}{2} )_+$. 
This finishes the proof.
\end{proof}

\begin{cor}[{\cite[Proposition~4.9]{NgRob16CommutatorsPureCa}}]\label{prp:PureDimDrop}
Every unital pure \ca{} $A$ admits unital \stHom{s} $Z_{n,n+1} \to A$ for every $n\in\NN$.
\end{cor}

We expect a positive answer to the following question.
Under the additional assumption of stable rank one, this is shown in \cite[Proposition~7.6.5]{AntPerThi18TensorProdCu}.

\begin{qst}
Does every unital pure \ca{} admit a unital \stHom{} from $\calZ$?
\end{qst}

We will use the following consequence of \cite[Proposition~1.12]{Kir06CentralSeqPI}, which was stated under the additional assumption that $A$ is unital in \cite[Proposition~2.2, Corollary~2.3]{KirRor15CentralSeqCharacters}.

\begin{prp}[Kirchberg]
\label{prp:IntoCommutingPosition}
Let $A$ be a separable \ca{}, and let $B,D \subseteq F(A)$ be separable sub-\ca{s} with $1_{F(A)} \in D$.
Then there exists a unital, injective \stHom{} $D \to B' \cap F(A)$.

Consequently, there exists a unital \stHom{} $\bigotimes^\infty_{\mathrm{max}} D \to F(A)$.
\end{prp}

\begin{thm}
\label{prp:CharZStable}
Let $A$ be a separable \ca.
Then the following are equivalent:
\begin{enumerate}[{\rm (1)}]
\item
The \ca{} $A$ is $\calZ$-stable.
\item
The (uncorrected) central sequence algebra $A' \cap A_\filter$ is pure.
\item
Kirchberg's central sequence algebra $F(A)$ is pure.
\end{enumerate}
\end{thm}
\begin{proof}
Assuming~(1), it follows from the implication `(1)$\Rightarrow$(4)' in \cref{prp:CharDStable2} that $A' \cap A_\filter$ is separably $\calZ$-stable, and therefore pure by \cref{prp:SepZStablePure}, thus verifying~(2).
The implication `(2)$\Rightarrow$(3)' follows since $F(A)$ is a quotient of $A' \cap A_\filter$, and since pureness passes to quotients by \cite[Theorem~4.11]{PerThiVil25arX:ExtPureCAlgs}.

Assuming~(3), let us verify~(1).
Since $F(A)$ is pure, there exists a unital \stHom{} $Z_{2,3} \to F(A)$ by \cref{prp:PureDimDrop}.
This is well-known to imply that~$A$ is $\calZ$-stable;
see, for example, \cite[Theorem~2.5]{KirRor15CentralSeqCharacters} in case $A$ is unital.
For completeness, we include the argument:
Using \cite[Theorem~1.1]{DadTom09ZStabInfTensProd} and \cref{prp:IntoCommutingPosition}, we obtain unital \stHom{s}
\[
\calZ \to 
\bigotimes_{\mathrm{min}}^\infty Z_{2,3} 
= \bigotimes_{\mathrm{max}}^\infty Z_{2,3} \to F(A).
\]
Applying the implication `(4)$\Rightarrow$(1)' in \cref{prp:CharDStable}, it follows that $A$ is $\calZ$-stable.
\end{proof}

\begin{rmk}
Analogous to \cref{prp:CharZStable}, one can show that $\calO_\infty$-stability of a separable \ca{} $A$ is encoded in the Cuntz semigroups of $A' \cap A_\filter$ and $F(A)$.
Specifically, the following are equivalent:
\begin{enumerate}
\item
The \ca{} $A$ is $\calO_\infty$-stable.
\item
The (uncorrected) central sequence algebra $A' \cap A_\filter$ is purely infinite.
\item
Kirchberg's central sequence algebra $F (A)$ is purely infinite.
\end{enumerate}

This relies on the fact that a separably $\calO_\infty$-absorbing \ca{} is purely infinite (analogous to \cref{prp:SepZStablePure}), and that a unital, purely infinite \ca{} admits a unital \stHom{} from $\calO_\infty$. 
\end{rmk}

\begin{thm}
\label{prp:CharSepZStable}
Let $A$ be a \ca.
Then the following are equivalent:
\begin{enumerate}[{\rm (1)}]
\item
The \ca{} $A$ is separably $\calZ$-stable.
\item
The relative commutant $B' \cap A_\filter$ is pure for every separable $B \subseteq A_\filter$.
\item
The quotient $(B' \cap A_\filter)/(B^\perp \cap A_\filter)$ is pure for every separable $B \subseteq A_\filter$.
\end{enumerate}
\end{thm}
\begin{proof}
We show the implications `(1)$\Rightarrow$(2)$\Rightarrow$(3)$\Rightarrow$(1)'.

(1)$\Rightarrow$(2):
Assume~(1).
To verify~(2), let $B \subseteq A_\filter$ be a separable sub-\ca{}. 
It follows from the implication `(1)$\Rightarrow$(3)' in \cref{prp:CharSepDStable} that $B' \cap A_\filter$ is separably $\calZ$-stable, and therefore pure by \cref{prp:SepZStablePure}.

\smallskip

(2)$\Rightarrow$(3):
This follows since pureness passes to quotients by \cite[Theorem~4.11]{PerThiVil25arX:ExtPureCAlgs}.

\smallskip

(3)$\Rightarrow$(1):
Assume~(3).
We will verify condition~(5) in \cref{prp:CharSepDStable}, which then implies that~$A$ is separably $\calZ$-stable.
Let $B \subseteq A_\filter$ be a separable sub-\ca{}. 
We need to find a unital \stHom{} $\calZ \to (B' \cap A_\filter)/(B^\perp \cap A_\filter)$.

Set $Q := (B' \cap A_\filter)/(B^\perp \cap A_\filter)$ and let $\pi \colon B' \cap A_\filter \to Q$ denote the quotient map.
We first show that $Q$ is $Z_{2,3}$-saturated.
Let $C \subseteq Q$ be a separable sub-\ca{}.
Choose a separable preimage, that is, a separable sub-\ca{} $C_0 \subseteq B' \cap A_\filter$ with $\pi(C_0)=C$.
Since the sub-\ca{} $E$ of $A_\filter$ generated by $C_0$ and $B$ is separable, it follows from the assumption that $(E' \cap A_\filter)/(E^\perp \cap A_\filter)$ is pure and therefore admits a unital \stHom{} from $Z_{2,3}$ by \cref{prp:PureDimDrop}.

The map $\pi$ induces a \stHom{} $E' \cap A_\filter \to C' \cap Q$ whose image contains the unit of $C' \cap Q$.
Since the annihilator of $C_0 \cup B$ is contained in the annihilator of $B$, this map factors through the quotient by $E^\perp \cap A_\filter$.
We obtain the desired map as the composition of the unital \stHom{s}
\[
Z_{2,3} 
\to (E' \cap A_\filter)/(E^\perp \cap A_\filter)
\to C' \cap (B' \cap A_\filter)/(B^\perp \cap A_\filter)
= C' \cap Q.
\]
This shows that $Q$ is $Z_{2,3}$-saturated.

Using this, we first find any unital \stHom{} $\varphi_1 \colon Z_{2,3} \to Q$, and then successively unital \stHom{s} $\varphi_n\colon Z_{2,3} \to Q$ for $n \geq 2$ such that
\[
\varphi_{n}(Z_{2,3}) \subseteq \big( \varphi_1(Z_{2,3}) \cup \cdots \cup \varphi_{n-1}(Z_{2,3}) \big)' \cap Q.
\]
This induces a unital \stHom{} $\bigotimes_{\mathrm{max}}^\infty Z_{2,3} \to Q$.
Using \cite[Theorem~1.1]{DadTom09ZStabInfTensProd}, we obtain unital \stHom{s}
\[
\calZ \to 
\bigotimes_{\mathrm{min}}^\infty Z_{2,3} 
= \bigotimes_{\mathrm{max}}^\infty Z_{2,3} \to Q. \qedhere
\]
\end{proof}

\section{Divisibility of central sequence algebras}
\label{sec:CentralDiv}

In this section, we show that $\calZ$-stability is equivalent to good central divisibility (that is, the central sequence algebra is \emph{almost divisible}). One could phrase this by saying that good central comparison is a consequence of good central divisibility.
In fact, $\calZ$-stability already follows if the central sequence algebra satisfies the weaker property of \emph{$n$-almost divisibility} for some $n$;
see \cref{prp:CentralDiv}.
It remains unclear if it suffices to assume the even weaker property of \emph{functional divisibility} in the central sequence algebras;
see \cref{qst:CentralFuncDiv}.


Given $N \in \NN$ with $N \geq 1$, we define the relation $\leq_N$ on a \CuSgp{} by setting $x \leq_N y$ if $nx \leq ny$ for all $n\geq N$.
We recall the definition of controlled comparison from \cite[Definition~3.6]{AntPerThiVil24arX:PureCAlgs};
see also \cite[Definition~6.1]{AntPerRobThi24TracesUltra}. For this, for any $[a]\in \Cu(A)$, we denote as costumary by $\widehat{[a]}$ the (lower semicontinuous) function defined on $QT(A)$ by $\widehat{[a]}(\tau)=d_\tau(a)$. 

\begin{dfn} 
\label{dfn:ContrComp}
A \ca{} $A$ is said to have \emph{controlled comparison} if for every $\gamma \in (0,1)$ and $d \geq 1$, there exists $N = N(\gamma, d)\geq 1$ such that the following statement holds:

For every $a,b \in M_d(A)_+$, if $\widehat{[a]} \leq \gamma \widehat{[b]}$ then $[a] \leq_N [b]$.
\end{dfn}

We will use the characterization of controlled comparison for \emph{stable} \ca{s} given below, which is an easy consequence of \cite[Corollary~7.6]{AntPerRobThi24TracesUltra}. Recall the relation $<_s$ from \cref{sec:purecentral}.

\begin{prp}
\label{prp:ContrCompStableCAlg}
Let $A$ be a stable \ca{}.~Then the following are equivalent:
\begin{enumerate}[{\rm (1)}]
\item
The \ca{} $A$ has controlled comparison.
\item
There exists $N\in\NN$ such that $\widehat{x} \leq \widehat{y}$ implies $x \leq Ny$, for all $x,y \in \Cu(A)$.
\item
There exists $M\in\NN$ such that $x <_s y$ implies $x \leq My$, for all $x,y \in \Cu(A)$.
\end{enumerate} 
\end{prp}
\begin{proof}
The equivalence of~(1) and~(2) is shown in \cite[Corollary~7.6]{AntPerRobThi24TracesUltra}.
Using that $x <_s y$ implies $\widehat{x} \leq \widehat{y}$, we deduce that~(2) implies~(3) with $M=N$.
Conversely, assume that~(3) holds with some $M\in\NN$.
We verify (2) for $N=2M$.
Let $x,y \in \Cu(A)$ satisfy $\widehat{x} \leq \widehat{y}$.
Then $\widehat{x} <_s \widehat{2y}$.
Let $x' \in \Cu(A)$ with $x' \ll x$.
By \cite[Theorem~5.2.18]{AntPerThi18TensorProdCu}, we deduce that $x' <_s 2y$, and thus $x' \leq 2My$ by assumption.
Since this holds for every element $x'$ way-below $x$, we deduce that $x \leq 2My$.
\end{proof}

\begin{prp}
\label{prp:ContrCompSufficient}
Let $A$ be a \ca{} (not necessarily stable).
Assume that there exists $M\in\NN$ such that $x <_s y$ implies $x \leq My$, for all $x,y \in \Cu(A)$.
Then~$A$ has controlled comparison.
\end{prp}
\begin{proof}
By \cref{prp:ContrCompStableCAlg}, we deduce that $A\otimes\Cpct$ has controlled comparison.
It follows easily from \cref{dfn:ContrComp} that a \ca{} has controlled comparison whenever its stabilization does.
\end{proof}

\begin{prp}
\label{prp:CentralDiv}
Let $A$ be a separable \ca.
Then the following are equivalent:
\begin{enumerate}[{\rm (1)}]
\item
The \ca{} $A$ is $\calZ$-stable.
\item
The algebra $A' \cap A_\filter$ is almost divisible.
\item
The unit of $F(A)$ is $n$-almost divisible for some $n \in \NN$.
\end{enumerate}
\end{prp}
\begin{proof}
We show the implications `(1)$\Rightarrow$(2)$\Rightarrow$(3)$\Rightarrow$(1)'.

(1)$\Rightarrow$(2):
If $A$ is $\calZ$-stable, then $A' \cap A_\filter$ is pure by \cref{prp:CharZStable}, and thus in particular almost divisible.

\smallskip

(2)$\Rightarrow$(3):
Assume that $A' \cap A_\filter$ is almost divisible.
Since $F(A)$ is a quotient of $A' \cap A_\filter$, and using that almost divisibility passes to quotients by \cite[Lemma~4.9]{PerThiVil25arX:ExtPureCAlgs}, it follows that $F(A)$ is almost divisible.
In particular, the unit of~$F(A)$ is $n$-almost divisible for some $n$ (namely $n=0$).

\smallskip

(3)$\Rightarrow$(1):
Let $n \in \NN$ be such that the unit in $F(A)$ is $n$-almost divisible.
We first verify the following claim:

\textbf{Claim:} \emph{Let $x,y \in \Cu(F(A))$ satisfy $x<_sy$.
Then $x \leq (n+1)y$.}

To prove the claim, pick $m \in \NN$ such that $(m+1)x \leq my$.
Using that $[1]$ is $n$-almost divisible, we obtain $z \in \Cu(F(A))$ such that
\[
mz \leq [1] \leq (n+1)(m+1)z.
\]

Since the elements $x$ and $y$, and the Cuntz subequivalence $(m+1)x \leq my$ are realized by countably many elements in $F(A)\otimes\Cpct$, we can find a separable, unital sub-\ca{} $B \subseteq F(A)$ such that $\Cu(B)$ contains elements $\overline{x},\overline{y}$ that are sent to $x$ and $y$ by the map $\Cu(B) \to \Cu(F(A))$ induced by the inclusion $B \to F(A)$, and such that $(m+1)\bar{x} \leq m\bar{y}$ in $\Cu(B)$.
Similarly, we can find a separable, unital sub-\ca{} $D \subseteq F(A)$ such that $\Cu(D)$ contains an element $\bar{z}$ such that $m\bar{z} \leq [1] \leq (n+1)(m+1)\bar{z}$ in $\Cu(D)$.

By \cref{prp:IntoCommutingPosition}, we may assume that $D$ is contained in $B' \cap F(A)$.
We obtain a unital \stHom{} $\pi \colon B \otimes_{\mathrm{max}} D \to F(A)$ such that $\pi(b \otimes 1)=b$ for $b \in B$.
As explained in \cite[Paragraph~6.4.10]{AntPerThi18TensorProdCu}, there exists a natural map $\varphi \colon \Cu(B)\times\Cu(D) \to \Cu( B \otimes_{\mathrm{max}} D )$ such that $\varphi([b],[d])=[b \otimes d]$ for $b \in B_+$ and $d \in D_+$.
We consider the composition $\psi := \Cu(\pi)\circ\varphi$, as shown in the following diagram:
\[
\xymatrix{
\Cu(B)\times\Cu(D) \ar[r]^-{\varphi} \ar@/_2pc/[rr]^{\psi}
& \Cu( B \otimes_{\mathrm{max}} D ) \ar[r]^-{\Cu(\pi)}
& \Cu( F(A) ).
}
\]
We note that $\psi$ is additive and order-preserving in each entry, and that $\psi(\bar{x},[1])=x$ and $\psi(\bar{y},[1])=y$.
We deduce that
\begin{align*}
x 
= \varphi(\bar{x}, [1])
&\leq \varphi(\bar{x},(n+1)(m+1)z) \\
&= \varphi((n+1)(m+1)\bar{x} , z) \\
&\leq \varphi((n+1)m\bar{y}, z) \\
&= \varphi((n+1)\bar{y} , mz)
\leq \varphi((n+1)\bar{y} , [1])
= (n+1)y,
\end{align*}
which proves the claim.


Now, it follows from \cref{prp:ContrCompSufficient} that $F(A)$ has controlled comparison.
By similar arguments as in the proof of the claim, one shows that the $n$-almost divisibility of the class of the unit in $F(A)$ propagates to all Cuntz classes, which means that $F(A)$ is $n$-almost divisible.
By \cite[Proposition~4.9]{AntPerThiVil24arX:PureCAlgs}, we get that that $F(A)$ is functionally divisible.
It then follows from \cite[Theorem~5.7]{AntPerThiVil24arX:PureCAlgs} that $F(A)$ is pure.
Applying \cref{prp:CharZStable}, we deduce that $A$ is $\calZ$-stable.
\end{proof}

In \cite{AntPerThiVil24arX:PureCAlgs}, the authors introduced \emph{functional divisibility} as a property weaker than $n$-almost divisibility for every $n$ and that, together with controlled comparison, still implies pureness.~With a view towards \cref{prp:CentralDiv}, the following question arises:

\begin{qst}
\label{qst:CentralFuncDiv}
Is a separable \ca{} $\calZ$-stable whenever its central sequence algebra is functionally divisible?
\end{qst}


\providecommand{\etalchar}[1]{$^{#1}$}
\providecommand{\bysame}{\leavevmode\hbox to3em{\hrulefill}\thinspace}
\providecommand{\noopsort}[1]{}
\providecommand{\mr}[1]{\href{http://www.ams.org/mathscinet-getitem?mr=#1}{MR~#1}}
\providecommand{\zbl}[1]{\href{http://www.zentralblatt-math.org/zmath/en/search/?q=an:#1}{Zbl~#1}}
\providecommand{\jfm}[1]{\href{http://www.emis.de/cgi-bin/JFM-item?#1}{JFM~#1}}
\providecommand{\arxiv}[1]{\href{http://www.arxiv.org/abs/#1}{arXiv~#1}}
\providecommand{\doi}[1]{\url{http://dx.doi.org/#1}}
\providecommand{\MR}{\relax\ifhmode\unskip\space\fi MR }
\providecommand{\MRhref}[2]{%
  \href{http://www.ams.org/mathscinet-getitem?mr=#1}{#2}
}
\providecommand{\href}[2]{#2}

\end{document}